\input louC.sty

\documentclass[a4paper,draft]{amsproc}
\usepackage{amsrefs}
\usepackage{amsmath}
\usepackage{mathrsfs}
\usepackage{amssymb}
\usepackage{amscd} 

\theoremstyle{plain}
 \newtheorem{thm}{\textbf{Theorem}}[section]

 \newtheorem{cor}{\textbf{Corollary}}[section]
\theoremstyle{definition}

\theoremstyle{remark}
 \newtheorem{rem}{\textbf{Remark}}[section]
 \numberwithin{equation}{section}

\renewcommand{\leq}{\leqslant}
\renewcommand{\geq}{\geqslant}

\title{A Note on the Spectral Area of Toeplitz Operators}

\subjclass[2010]{Primary 30,47}

\author[]{Cheng Chu}
\address{
Department of Mathematics \\ 
Washington University in Saint Louis  \\ 
Saint Louis, Missouri}
\email{chengchu@math.wustl.edu}

\author[]{Dmitry Khavinson}
\address{
Department of Mathematics \\
University of South Florida \\
Tampa, Florida}
\email{dkhavins@usf.edu}

\begin{document}
\setcounter{page}{1}
\thispagestyle{empty}

\begin{abstract}
In this note, we show that for hyponormal Toeplitz operators, there exists a lower bound for the area of the spectrum. This extends the known estimate for the spectral area of Toeplitz operators with an analytic symbol.
\end{abstract}

\maketitle

\section{Introduction}  

Let $\DD$ be the open unit disk in the complex plane. Let $L^2$ denote the Lebesgue space of square integrable functions on the unit circle $\partial\DD$. The Hardy space $H^2$ is the subspace of $L^2$ of analytic functions on $\DD$. Let $P$ be the orthogonal projection from $L^2$ to $H^2$.
For $\varphi\in L^\infty$, the space of bounded Lebesgue measurable functions on $\partial\DD$, the Toeplitz operator $T_{\varphi}$ and the Hankel operator $H_\varphi$ with symbol $\varphi$ are defined on $H^2$ by $$T_{\varphi}h=P(\varphi h),$$ and
\beq\label{Han}
H_{\varphi}h=U(I-P)(\varphi h),
\eeq
for $h\in H^2$.
Here $U$ is the unitary operator on $L^2$ defined by $$Uh(z)=\bz h (\bz).$$

Recall that the spectrum of a linear operator $T$, denoted as $sp(T)$, is the set of complex numbers $\lambda$ such that $T - \lambda I$ is not invertible; here $I$ denotes the identity operator. Let $[T^*,T]$ denote the operator $T^*T-TT^*$, called the self-commutator of $T$. An operator $T$ is called hyponormal if $[T^*,T]$ is positive.
Hyponormal operators satisfy the celebrated Putnam inequality \cite{put70}
\begin{thm}\label{P}
If $T$ is a hyponormal operator, then
$$
\|[T^*,T]\|\leq \frac{Area(sp(T))}{\pi}.
$$
\end{thm}

Notice that a Toeplitz operator with analytic symbol $f$ is hyponormal, and it is well known that $sp(T_f)=\overline{f(\DD)}$. The lower bounds of the area of $sp(T_f)$ were obtained in \cite{kh(T)} (see \cite{ax85},\cite{ale72} \cite{stan86} and \cite{stan89} for generalizations to uniform algebras and further discussions). Together with Putnam's inequality such lower bounds were used to prove the isoperimetric inequality (see \cite{ben14},\cite{ben06} and the references there). Recently, there has been revived interest in the topic in the context of analytic Topelitz operators on the Bergman space (cf. \cite{bell14}, \cite{ol14} and \cite{fl15}). Together with Putnam's inequality, the latter lower bounds have provided an alternative proof of the celebrated St. Venant's inequality for torsional rigidity.

In the general case, Harold Widom \cite{wid64} proved the following theorem for arbitrary symbols.
\begin{thm}
Every Toeplitz operator has a connected spectrum.
\end{thm}
The main purpose of this note is to show that for a rather large class of Topelitz operators on $H^2$, hyponormal operators with a harmonic symbol, there is
still a lower bound for the area of the spectrum, similar to the lower bound obtained in \cite{kh(T)} in the context of uniform algebras.

We shall use the following characterization of the hyponormal Toeplitz operators given by Cowen in \cite{cow88}
\begin{thm}\label{cow}
Let $\varphi\in L^\infty$, where $\varphi=f+\bar{g}$ for $f$ and $g$ in $H^2$. Then $T_{\varphi}$ is hyponormal if and only if
$$
g=c+T_{\bh}f,
$$
for some constant $c$ and $h\in H^\infty$ with $\|h\|_{\infty}\leq 1.$
\end{thm}

\section{Main Results}
In this section, we obtain the lower bound for the area of the spectrum for hyponormal Toeplitz operators by estimating the self-commutators.

\begin{thm}\label{2.1}
Suppose $\varphi \in L^\infty$ and $$\varphi=f+ \overline{T_{\bar{h}}f} ,$$ for $f,h\in H^\infty$, $\|h\|_{\infty}\leq 1$ and $h(0)=0$. Then
$$
\|[T^{*}_{\varphi}, T_{\varphi}]\|\geq \int |f-f(0)|^2\frac{d\theta}{2\pi}=||P(\varphi)-\varphi(0)||^2_2.
$$
\end{thm}
\begin{proof}
Let
\beq\label{g} g=T_{\bar{h}}f.\eeq
For every $p$ in $H^2$,
\begin{align*}
\langle[T^{*}_{\varphi}, T_{\varphi}]p,p\rangle&=\langle T_{\varphi}p, T_{\varphi}p\rangle-\langle T^*_{\varphi}p, T^*_{\varphi}p\rangle\\
&=\langle fp+P(\bg p), fp+P(\bg p)\rangle-\langle gp+P(\barf p), gp+P(\barf p)\rangle\\
&=||fp||^2-||P(\barf p)||^2-||gp||^2+||P(\bg p)||^2\\
&=||\barf p||^2-||P(\barf p)||^2-||\bg p||^2+||P(\bg p)||^2\\
&=||H_{\barf} p||^2-||H_{\bg}p||^2,
\end{align*}
where $||\cdot||$ means the $||\cdot||_{L^2(\partial \DD)}$.
The third equality holds because
$$
\langle fp, P(\bg p)\rangle=\langle fp, \bg p\rangle=\langle gp, \barf p\rangle=\langle gp, P(\barf p)\rangle.
$$
By the computation in \cite{cow88}*{p. 4}, \eqref{g} implies $$H_{\bg}=T_{\bk}H_{\barf},$$
where $k(z)=\overline{h(\bz)}$. Thus
\beq\label{rmk}
\langle[T^{*}_{\varphi}, T_{\varphi}]p,p\rangle=||H_{\barf} p||^2-||T_{\bk}H_{\barf}p||^2,
\eeq
for $k\in H^\infty$, $\|k\|_{\infty}\leq 1$ and $k(0)=0$.

First, we assume $k$ is a Blaschke product vanishing at $0$. Then
$$|k|=1\,\,\m{on}\,\, \partial\DD.$$
Let $u=H_{\barf} p\in H^2$. By \eqref{rmk} we have
\begin{align}\label{T}
\langle[T^{*}_{\varphi}, T_{\varphi}]p,p\rangle&=||u||^2-||T_{\bk} u||^2=||u||^2-||\bk u||^2+||H_{\bk}u||^2=||H_{\bk}u||^2.
\end{align}
Then
\begin{align*}
||H_{\bk}u||&=||(I-P)(\bk u)||=||k\bu-\overline{P(\bk u)}||\\
&\geq \sup_{\substack{m\in H^2\\m(0)=0}}\frac{|\langle k\bu-\overline{P(\bk u)},m\rangle|}{||m||}\\
&={\sup_{\substack{m\in H^2\\m(0)=0}}{1 \over ||m||}\left|\int k\bu \bm \,{d\theta\over 2\pi}\right|}.
\end{align*}
The last equality holds because $m(0)=0$ implies that $\bm$ is orthogonal to $H^2$.
Since $k(0)=0$, taking $m=k$, we find
\beq\label{u}
||H_{\bk}u||\geq \left|\int\bu \,{d\theta\over 2\pi}\right|=|u(0)|.
\eeq

Next, suppose $k$ is a convex linear combination of Blaschke products vanishing at 0, i.e.
$$
k=\Ga_1B_1+\Ga_2B_2+...+\Ga_lB_l,
$$
where $B_j$'s are Blaschke products with $B_j(0)=0$, $\Ga_j\in[0,1]$ and $\displaystyle\sum_{j=1}^l \Ga_j=1$.

By \eqref{T} and \eqref{u}, for each $j$
\begin{align*}
&||u||^2-||T_{\bar{B_j}}u||^2=||H_{\bar{B_j}}u||^2\geq |u(0)|^2\\
\Longrightarrow\, &|| T_{\bar{B_j}}u||\leq \sqrt{||u||^2-|u(0)|^2}=||u-u(0)||.
\end{align*}
Then
\begin{align}
\label{M}||u||^2-||T_{\bk} u||^2&=||u||^2-||\Ga_1T_{\bar{B_1}}u+\Ga_2T_{\bar{B_2}}u+...+\Ga_lT_{\bar{B_l}}u||^2\\
\nnb&\geq ||u||^2-\Big(\Ga_1||T_{\bar{B_1}}u||+\Ga_2||T_{\bar{B_2}}u||+...+\Ga_l||T_{\bar{B_l}}u||\Big)^2\\
\nnb&\geq ||u||^2- ||u-u(0)||^2=|u(0)|^2.
\end{align}

In general, for $k$ in the closed unit ball of $H^\infty$, vanishing at $0$, by Carath\'eodory's Theorem(cf. \cite{gar81}*{p. 6}), there exists a sequence $\{B_n\}$ of finite Blaschke products such that
$$B_n\longrightarrow k\q \m{pointwise on}\,\,\DD.$$
Since $B_n$'s are bounded by $1$ in $H^2$, passing to a subsequence we may assume
$$B_n\longrightarrow k\q \m{weakly in}\,\,H^2.$$
Then by \cite{rud91}*{Theorem 3.13}, there is a sequence $\{k_n\}$ of convex linear combinations of Blaschke products such that
$$k_n\longrightarrow k\q \m{in}\,\,H^2.$$
Since $k(0)=0$, we can let those $k_n$'s be convex linear combinations of Blaschke products vanishing at $0$.

Then
$$
||T_{\bk_n}u-T_{\bk} u||=||P(\bk_n u-\bk u)||\leq||k_n-k||\cdot||u||\to 0.
$$
Since \eqref{M} holds for every $k_n$, we have
\begin{align*}
\langle[T^{*}_{\varphi}, T_{\varphi}]p,p\rangle&=||u||^2-||T_{\bk}u||^2=\lim_{n\to\infty}(||u||^2-||T_{\bk_n}u||^2)\\
&\geq|u(0)|^2=|(H_{\barf} p)(0)|^2.
\end{align*}

By the definition of Hankel operator \eqref{Han},
\begin{align*}
|(H_{\barf} p)(0)|&=|\langle p\barf, \bz \rangle|=\left|\int \barf zp\,{d\theta\over 2\pi}\right|.
\end{align*}
From the standard duality argument (cf. \cite{gar81}*{Chapter IV}), we have
\begin{align*}
\sup_{\substack{||p||=1 \\ p\in H^2}}\left|\int \barf zp\,{d\theta\over 2\pi}\right|&=\sup\bigg\{ \left|\int \barf p\,{d\theta\over 2\pi}\right| : p\in H^2, ||p||=1, p(0)=0 \bigg\}\\
&=\m{dist}(\barf, H^2)=||f-f(0)||.
\end{align*}
Hence
$$
||[T^{*}_{\varphi}, T_{\varphi}]||=\sup_{\substack{||p||=1 \\ p\in H^2}}|\langle[T^{*}_{\varphi}, T_{\varphi}]p,p\rangle|\geq ||f-f(0)||^2.
$$
\end{proof}
\begin{rem}
For arbitrary $h$ in the closed unit ball of $H^\infty$, it follows directly from \eqref{rmk} that $T_{\varphi}$ is normal if and only if $h$ is a unimodular constant. So we made the assumption that $h(0)=0$ to avoid these trivial cases. Of course, Theorem \ref{2.1} implies right away that $T_\varphi$ is normal if and only if $f=f(0)$, i.e., when $\varphi$ is a constant, but under more restrictive hypothesis that $h(0)=0$.
\end{rem}

Applying Theorem \ref{P} and \ref{cow}, we have
\begin{cor}
Suppose $\varphi \in L^\infty$ and $$\varphi=f+ \overline{T_{\bar{h}}f} ,$$ for $f,h\in H^\infty$, $\|h\|_{\infty}\leq 1$ and $h(0)=0$. Then
$$
Area(sp(T_\varphi))\geq\pi||P(\varphi)-\varphi(0)||^2_2.
$$
\end{cor}

\begin{rem}
Thus, the lower bound for the spectral area of a general hyponormal Toeplitz operator $T_\varphi$ on $\partial \DD$ still reduces to the $H^2$ norm of the analytic part of $\varphi$. For analytic symbols this is encoded in \cite{kh(T)}*{Theorem 2} in the context of Banach algebras. In other words, allowing more general symbols does not reduce the area of the spectrum.
\end{rem}

\section*{Acknowledgments}
The first author gratefully acknowledges the hospitality of the Department of Mathematics and Statistics
at the University of South Florida during the work on the paper.

\bibliography{references}

\end{document}